\documentclass[12pt]{amsart}

\usepackage[utf8]{inputenc}
\usepackage[shortlabels]{enumitem}
\usepackage{microtype}
\usepackage{subcaption}

\usepackage{amsmath}
\usepackage{amssymb}
\usepackage{amsfonts}
\usepackage{latexsym}
\usepackage{stmaryrd}
\usepackage{mathtools}
\usepackage{mathrsfs}
\usepackage{bbm}

\DeclareFontFamily{OT1}{pzc}{}
\DeclareFontShape{OT1}{pzc}{m}{it}%
{<-> s * [1.15] pzcmi7t}{}
\DeclareMathAlphabet{\mathzc}{OT1}{pzc}{m}{it}
\DeclareMathSymbol{\qm}{\mathalpha}{operators}{"3F}

\usepackage{tikz-cd}
\usepackage{quiver}

\newcommand{\cH}{\star}
\newcommand{\cV}{\cdot}

\newcommand{\Id}[1]{\mathrm{Id}_{#1}}
\newcommand{\id}[1]{#1}%
\newcommand{\iid}[1]{\mathrm{id}_{#1}}
\newcommand{\?}{\mathbf{\qm}}

\newcommand{\mc}[1]{\ensuremath{\mathzc{#1}}} 
\newcommand{\mC}[1]{\ensuremath{\operatorname{\mathzc{#1}}}} 
\newcommand{\mFc}[1]{\ensuremath{\operatorname{#1}}} 
\newcommand{\mb}[1]{\mathbbm{#1}} 

\newcommand{\fc}[3]{{#1}: {#2} \rightarrow {#3}}
\newcommand{\nt}[3]{{#1}: {#2} \Rightarrow {#3}}
\newcommand{\morf}[3]{{#2} \xrightarrow{#1} {#3}}

\newcommand{\ldual}[1]{{#1}^{\ell}}
\newcommand{\rdual}[1]{{#1}^{\mathit{r}}}

\newcommand{\Hom}[3][]{\mFc{Hom}_{#1} \left( #2 , #3 \right)}
\newcommand{\iHom}[2]{ \left[ #1 , #2 \right]}

\newcommand{\liHom}[2]{\ldual{\iHom{#1}{#2}}}
\newcommand{\riHom}[2]{\rdual{\iHom{#1}{#2}}}

\newcommand{\curry}[1]{\ulcorner #1 \urcorner}
\newcommand{\uncurry}[1]{\llcorner #1 \lrcorner}

\usepackage{amsmath}
\usepackage{amsthm}

\newtheorem{thm}{Theorem}[section]
\newtheorem{lem}[thm]{Lemma}
\newtheorem{prop}[thm]{Proposition}
\newtheorem{cor}[thm]{Corollary}

\theoremstyle{definition}
\newtheorem{defi}[thm]{Definition}

\theoremstyle{remark}
\newtheorem{ex}[thm]{Example}

\newtheorem{rem}[thm]{Remark}

\newtheorem{nota}[thm]{Notation}

\newcommand{\defTerm}[1]{\textsl{#1}}
\usepackage[colorlinks=true]{hyperref}
\hypersetup{allcolors=[rgb]{0.1,0.1,0.4}}

\date{}
\author{Johan Felipe García Vargas}
\title {Galois correspondence for augmented monads}
\address{Departamento de Matemáticas, Universidad de los Andes,
Bogotá, Colombia.
}
\email{jf.garcia14@uniandes.edu.co}
\keywords{
  Augmented monads;
  Grothendieck's Galois Theory;
  Galois connection;
  Hopf monads;
  Tannaka duality
}
\subjclass{18C15, 18E50, 18M05, 18D15, 16T05}

\begin{document}

\begin{abstract}
We establish a Galois connection between
sub-monads of an augmented monad and
sub-functors of the forgetful functor
from its Eilenberg-Moore category.
This connection is given in terms of
invariants and stabilizers defined through universal properties.
An explicit procedure for the computation of invariants
is given assuming the existence of suitable right adjoints.
Additionally, in the context of monoidal closed categories,
a characterization of stabilizers is made
in terms of Tannakian reconstruction.
\end{abstract}

\maketitle
\setcounter{tocdepth}{1}
\tableofcontents

\section{Introduction} \label{sec:intro}

The structural approach to Galois theory, as develop by Emil Artin and taught in
most algebra textbooks, place as fundamental theorem the correspondence, given
by invariants and stabilizers, between subgroups of symmetries and intermediate
structures.

In \cite[Exposé V]{SGA1}, Grothendieck reframed Galois theory as an equivalence
between a category of interest and the category of actions of a group.
Since the actions of a group \(G\) form the
Eilenberg-Moore category for the monad \(G \times \?\),
Grothendieck's formulation is a monadicity result.
This article proposes a path for recovering
the Galois correspondence from a monadicity result.

Our main hypothesis is that the monad possess
an augmentation. Where, an \defTerm{augmentation}
is a monad homomorphism to the identity of the base category.
This hypothesis is strong, because in the case of Hopf monads
it implies that the monad is \(\otimes\)-representable,
as shown in \cite[Section 5]{BrLaVi-HMMC}.

Morally, having an augmentation distinguish groups from groupoids.
More precisely, the monad induced by a groupoid has an augmentation
only when every morphism in the groupoid is an endomorphism.
However, this makes sense, because there is no clear meaning for
invariants of a groupoid action when it moves elements from a set
to another.

Before presenting the details, let me exhibit the intuition behind them.
Imagine the monad \(T\) as an algebraic gadget useful for ``capturing the
symmetries'' of some kind of structure defined over a category \(\mc{C}\).
With this in mind, we regard the forgetful functor \(U\)
as the ``universal \(T\)-structure''.
A Galois correspondence will emerge from
comparing the action
of a part of the \emph{symmetries},
parametrized by a monad homomorphism \(\nt{h}{S}{T}\),
over a part of the of the \emph{structure},
parametrized by a natural transformation \(\nt{\alpha}{V}{U}\),
against a trivial action induced by an augmentation of the monad.

The next example illustrates the notions that we aim to generalize and
might be considered a trailer for the paper:

\begin{ex} \label{ex:Inv-Stab-Group}
  Let \(G\) be a group. Consider the augmented monad \(G \times \?\) in
  \(\mC{Sets}\), where the augmentation is the projection to the second
  component. The Eilenberg-Moore category \(\mC{Sets}^{G}\) is the category of
  \(G\)-sets and equivariant maps.

  For any group homomorphism \(\morf{\phi}{H}{G}\) the \(H\)-invariants define a
  subfunctor of the forgetful \(\fc{U}{\mC{Sets}^{G}}{\mc{Sets}}\), mapping an action
  \((X, \morf{r}{G \times X}{X})\) to the set inclusion
  \[\mFc{Inv} H_{\phi} (X,r) = \{ x \in X : \; \phi(h) x = x \;
    \text{for all} \; h \in H \} \subseteq X.\]
  We  will dissect this construction in the following way:

  The action \(\morf{r}{G \times X}{X}\) and the augmentation \(\morf{e \times X }{G \times X}{X}\)
  (the \(e\) stands for erase, or for the constant function to the group identity)
  parametrize for every \(x \in X\) functions from \(G\) to \(X\),
  mapping \(g \mapsto (r(g,x) = g x)\) and \(g \mapsto (e(g) x = x)\), respectively.

  Denoting by \(\morf{\iHom{\phi}{X}}{\iHom{G}{X}}{\iHom{H}{X}}\) the function precomposing by \(\morf{\phi}{H}{G}\),
  we can present \(\mFc{Inv} H (X,r) \) as the equalizer of
\(
\begin{tikzcd}[cramped]
  {X} & {\iHom{G}{X}} & {\iHom{H}{X}}.
  \arrow["{\curry{r}}", shift left=1.5, from=1-1, to=1-2]
  \arrow["{\curry{e \otimes X}}"', shift right=1, from=1-1, to=1-2]
  \arrow["{\iHom{\phi}{X}}", from=1-2, to=1-3]
\end{tikzcd}
\)
This presentation of invariants is generalized in
Lemma~\ref{lem:inv-monad-adjoint}.

  On the other hand, for every natural transformation \(\nt{\alpha}{V}{U}\)
  the stabilizer of \(V_{\alpha}\) is the subgroup of \(G\) defined by
  \[ \mFc{Stab} V_{\alpha} = \{g \in G : \; g \alpha_{M}(v) = \alpha_{M}(v) \;
    \text{for every}\; M \in \mFc{Sets}^{G} \;\text{and every}\; v \in VM \}.\]
  Inspired by Tannakian reconstruction, we can also dissect the construction of stabilizers.
  Notice that each \(g \in G\) gives a natural endomorphism of \(U\) obtained by acting,
  explicitly \(\nt{\rho(g)}{U}{U}\) is defined by \(\rho(g)_{M} (x) = r_{M}(g,x) = g x\)
  for every \(\mFc{G}\)-set \(M = (X,r_{M})\) and every \(x\in X\).
  The augmentation similarly defines a natural endomorphism of \(U\) for each \(g \in G\),
  \(\nt{\epsilon(g)}{U}{U}\) defined by \(\epsilon(g)_{M}(x) = e(g) x = x \)
  for every \(\mFc{G}\)-set \(M = (X,r_{M})\) and every \(x\in X\),
  which is clearly the identity of \(U\).

  Denoting by \(\morf{\mFc{Nat} ({\alpha}, {U})}{\mFc{End} (U)}{\mFc{Nat}(V, U) }\) the function precomposing by \(\nt{\alpha}{V}{U}\),
  we can present \(\mFc{Stab} V_{\alpha} \) as the equalizer of
\(
\begin{tikzcd}[cramped]
  {G} & {\mFc{End}(U)} & {\mFc{Nat}(V, U)}.
  \arrow["{\rho}", shift left=1.5, from=1-1, to=1-2]
  \arrow["{\epsilon}"', shift right=1, from=1-1, to=1-2]
  \arrow["{\mFc{Nat}(\alpha, V)}", from=1-2, to=1-3]
\end{tikzcd}
  \)
  This presentation of stabilizers is generalized in
  Lemma~\ref{lem:Stab-as-equal}.

  The well known Galois connection between invariants and stabilizers is
  generalized in Corollary~\ref{cor:Galois-corr}.
\end{ex}

The structure of the document is straightforward.
Section \ref{sec:preliminars} contains three preliminars
one about lifting functors to Eilenberg-Moore categories,
one about monoidal closed categories and Hopf monads
and one about Galois connections.
Section \ref{sec:Inv-Stab} introduces invariants and
stabilizers as universal representants for the fix relation.
Section \ref{sec:Invariants-as-adjoints} is about computing
invariants and Section \ref{sec:Stab-mono} about computing
stabilizers in monoidal closed categories.

\section{Preliminars} \label{sec:preliminars}

\begin{nota}
  We write compositions from right to left. Given natural transformations
  \(\nt{\alpha}{F}{F'}\) and \(\nt{\gamma}{F'}{F''}\), where
  \(\fc{F,F',F''}{\mc{C}}{\mc{D}}\), we denote their \defTerm{vertical composition}
  by \(\nt{\gamma \cV \alpha}{F}{F''}\). If additionally given
  \(\nt{\beta}{G}{G'}\), with \(\fc{G,G'}{\mc{D}}{\mc{E}}\), we denote the
  \defTerm{horizontal composition} by \(\nt{\beta \cH \alpha}{GF}{G'F'}\). We will
  abbreviate identities of functors (and of objects) as $\iid{G} = \id{G}$. We
  will also omit the $\cH$ when \defTerm{whiskering}. For instance,
  \( \beta \cH \alpha = (\id{G'} \alpha) \cV (\beta \id{F}) = (\beta \id{F'}) \cV (\id{G} \alpha). \)
\end{nota}

\subsection{Monadic lifting to Eilenberg-Moore categories} \label{sec:monadic-lift}

It is well known, see \cite[Chapter VI]{MacLaneCategories} for instance, that
every adjunction yields a monad and that the Eilenberg-Moore construction allow
to decompose a monad into an adjunction in a ``universal'' way. Where by
universal, we mean that every other adjunction, yielding the same monad, is
included in it. It is less well known, that the lift of functors to
Eilenberg-Moore categories can be understood in terms of natural
transformations of the base categories. We will make a quick review of this and
refer the reader to \cite[Chapter 2]{boehm18_hopf_algeb_their_gener_categ} for
the details.

We denote monads by \((T, \mu, \eta)\) and adjunctions by \((\eta, L \dashv R, \varepsilon)\) as displayed in the next figure:

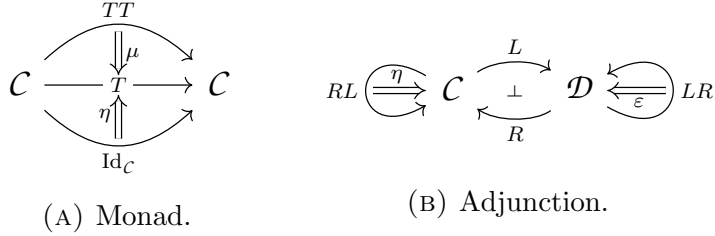
\begin{figure}[h]\label{fig:monad-adjoint} \begin{subfigure}{.4\linewidth}
    \[
  \begin{tikzcd}
    \mc{C} & & \mc{C}
    \arrow[""{name=0, anchor=center, inner sep=0}, "T"{description}, from=1-1, to=1-3]
    \arrow[""{name=1, anchor=center, inner sep=0}, "{\Id{\mc{C}}}"', shift right=1, curve={height=18pt}, from=1-1, to=1-3]
    \arrow[""{name=2, anchor=center, inner sep=0}, "TT", shift left=1, curve={height=-18pt}, from=1-1, to=1-3]
    \arrow["\eta", shorten <=2pt, shorten >=3pt, Rightarrow, from=1, to=0]
    \arrow["\mu", shorten <=2pt, shorten >=3pt, Rightarrow, from=2, to=0]
  \end{tikzcd}
    \]
    \caption{Monad.}
  \end{subfigure}
  \begin{subfigure}{.4\linewidth}
    \[
  \begin{tikzcd}
    \mc{C} & \mc{D}
    \arrow[""{name=0, anchor=center, inner sep=0}, "L", curve={height=-12pt}, from=1-1, to=1-2]
    \arrow[""{name=1, anchor=center, inner sep=0}, "R", curve={height=-12pt}, from=1-2, to=1-1]
    \arrow[""{name=3, anchor=center, inner sep=0}, "RL"', loop, out=150, in=210,looseness=8, from=1-1, to=1-1]
    \arrow[""{name=4, anchor=center, inner sep=0}, "LR"', loop, out=-30, in=30,  looseness=8, from=1-2, to=1-2]
    \arrow["\dashv"{anchor=center, rotate=-90}, draw=none, from=1, to=0]
    \arrow["\eta", shorten <=2pt, Rightarrow, from=3, to=1-1]
    \arrow["\varepsilon", shorten <=2pt, Rightarrow, from=4, to=1-2]
  \end{tikzcd}
    \]
    \caption{Adjunction.}
  \end{subfigure}
    \caption{Structure of monads and adjunctions.}
\end{figure}

\begin{defi}\label{def:Eilenberg-Moore}
  Let \((T, \mu, \eta)\) be a monad over a category $\mc{C}$.
  The \defTerm{Eilenberg-Moore category} of $T$,
  denoted $\mc{C}^T$, has $T$-actions as objects
  and $T$-morphisms as morphisms.

  \begin{enumerate}[(a)]
    \item A $T$-\defTerm{action} is a pair $M = (X, \morf{r_{M}}{TX}{X})$, where $X$ is in $\mc{C}$ and $r$ satisfies:
    \begin{align*}
      \text{associativity}& &r_{M} \mu_{X} &= r Tr_{M}, \\
      \text{and unity,}&    &r_{M} \eta_X  &= \id{X}.
    \end{align*}

    \item A $T$-\defTerm{morphism} $\morf{f}{M}{N}$,
          between $T$-actions $M = (X,r)$ and $N = (Y,r')$,
          is a morphism $\morf{f}{X}{Y}$ satisfying $ fr = r'Tf$.

    \item The \defTerm{forgetful functor} $\fc{U^T}{\mc{C}^T}{\mc{C}} $, defined by
      \[U^T (\morf{f}{(X,r)}{(Y,r')}) = (\morf{f}{X}{Y}),\]
      has a left adjoint $\fc{F^T}{\mc{C}}{\mc{C}^T}$.
      This functor is known as the \defTerm{free functor}  and
      it is defined by
      \[F^T X = ( TX , \morf{\mu_{X}}{TTX}{TX} ) \quad \text{ and } \quad F^T (\morf{f}{X}{Y}) = \morf{Tf}{(TX,\mu_X)}{(TY,\mu_Y)}.\]

      Notice that $T = U^T F^T$, hence the unit of the adjunction
      $\nt{\eta}{\Id{\mc{C}}}{U^T F^T}$ is equal to the unit of the monad $\nt{\eta}{\Id{\mc{C}}}{T}$.
      Moreover, the evaluation $\nt{\varepsilon}{F^T U^T}{\Id{\mc{C}^T}}$ is given by $\varepsilon_M = \morf{r}{(TX,\mu_X)}{(X,r)}$,
      for any $T$-action $M = (X, \morf{r}{TX}{X})$.

    \item Suppose that there is another adjunction $(\eta, L \dashv R, \upsilon)$
          with $\fc{R}{\mc{D}}{\mc{C}}$ such that $T = R L$ and $\mu = R \upsilon L$.
      The \defTerm{comparison functor} $\fc{K}{\mc{D}}{\mc{C}^T}$ is defined by
          \[K A = (R A, \morf {R \upsilon_A}{R L R A }{ R A } ),\]
      for every object $A$ in $\mc{D}$.
      It is easy to verify that $K A$ is a $T$-action and obviously $U^T K = R$.

    \item A functor $\fc{R}{\mc{D}}{\mc{C}}$ is called \defTerm{monadic} if it has a left adjoint , $L \dashv R$,
          and the comparison functor is an equivalence
      from \(\mc{D}\) to the Eilenberg-Moore category \(\mc{C}^{T}\)
          of the induced monad $T = R L$.
  \end{enumerate}
\end{defi}
\begin{defi}\label{def:Lifting-data}
  Let $(T,\mu, \eta)$ and $(T',\mu', \eta')$ be monads over $\mc{C}$ and
  $\mc{C}'$, respectively. Given a functor $\fc{G}{\mc{C}}{\mc{C}'}$, we say
  that:
  \begin{itemize}
    \item A \defTerm{lift of} $G$ is a functor $\fc{\overline{G}}{\mc{C}^T}{\mc{C}'^{T'}}$
          such that $U^{T'} \overline{G} = G U^{T} $.

    \item A \defTerm{lifting data for} $G$ is a natural transformation
          $\nt{\gamma}{T' G}{G T}$ such that
    \begin{align*}
      & &\gamma \cV (\mu' \id{G}) &= (\id{G} \mu) \cV (\gamma \id{T}) \cV (\id{T'} \gamma),\\
      &\text{  and } & \gamma \cV (\eta' \id{G}) &= \id{G} \eta.
    \end{align*}
    \item When $\mc{C}=\mc{C}'$, an \defTerm{homomorphism of monads}
      $\nt{h}{T'}{T}$ is a lifting data for $\Id{\mc{C}}$. In that case, the above
      axioms simplify to $h \cV \mu' = \mu \cV (h \cH h)$ and $h \cV \eta' = \eta$.
  \end{itemize}
\end{defi}

\begin{thm} \label{thm:lifts-eq-lifting-data}
  {\cite[Theorem and definition 2.27]{boehm18_hopf_algeb_their_gener_categ}}
  Let $T$, $T'$ and $G$ as in the previous definition.
  There is a biyection between lifts of $G$ and lifting data for $G$.
\end{thm}

As a special case of the above theorem, taking the identity as base functor (i.e
\(\mc{C}=\mc{C}'\) and \(G=\Id{\mc{C}}\)), we obtain the following corollary:

\begin{cor}\label{cor:Monad-hom-funtors}
  Let \(T\) and \(T'\) be monads over \(\mc{C}\). There is a biyection between monad
  homomorphisms \(\nt{h}{T'}{T}\) and functors \(\fc{H}{\mc{C}^T}{\mc{C}^{T'}}\)
  such that \(U^{T'} H = U^T\).
\end{cor}

The next proposition gives a criterion for lifting natural transformations:

\begin{prop}\label{prop:lifts_nts}
  {\cite[Theorem and definition 2.30]{boehm18_hopf_algeb_their_gener_categ}}
  Given a natural transformation $\nt{\omega}{G}{H}$ between functors from
  $\mc{C}$ to $\mc{C'}$ with lifting data $\nt{\alpha}{T'G}{GT}$ and
  $\nt{\beta}{T'H}{HT}$ along the monads $T$ and $T'$.
  There is an unique $\nt{\bar{\omega}}{G^{\alpha}}{H^{\beta}}$ which lifts
  $\omega$ (i.e. $\id{U^{T'}} \bar{\omega} = \omega \id{U^T} $), if and only if,
  $(\omega \id{T}) \cV \alpha = \beta \cV (\id{T'} \omega)$.
\end{prop}

\subsection{Monoidal closed categories and Hopf monads} \label{sec:Hopf-monads}
A \defTerm{monoidal category} is a category \mc{C} endowed with
a functor \(\fc{\otimes}{\mc{C}\times\mc{C}}{\mc{C}}\) called \defTerm{monoidal product},
an object $\mb{1}$ in $\mc{C}$ called \defTerm{unit object},
and natural isomorphisms called associativity and unit constrains
which are subject to well known axioms of coherence.
In virtue of the strictness theorems,
--- see \cite[Lemma 3.7 and 3.8]{boehm18_hopf_algeb_their_gener_categ}, \cite[Section VII.2]{MacLaneCategories} or \cite[Section 2.8]{Etin-TC}  ---
it is unnecessary to explicitly keep track of
the associative or unit constraints in any monoidal category,
as long as you are only interested in properties up to
monoidal equivalence of categories.

One key feature of monoidal categories is that an object \(X\),
can be regarded as an endofunctor \(X \otimes \?\);
this kind of functors are called \(\otimes\)-\defTerm{representable}.
The same idea applies to morphisms and natural transformations.
For instance, \(\otimes\)-representable monads correspond to monoids.
Furthermore, the Eilenberg-Moore category of such a monad is
the category of actions of the monoid.

Bimonads are monads which lift the monoidal product to the
Eilenberg-Moore category.
More precisely, a monad \(T\) over \((\mC{C}, \otimes. \mb{1})\) is a \defTerm{bimonad}
if and only if \(C^{T}\) is also monoidal and the forgetful functor
\(\fc{U^{T}}{\mc{C}^{T}}{\mc{C}}\) is strong monoidal.
By means of Theorem \ref{thm:lifts-eq-lifting-data},
Bimonads are characterized as opmonoidal monads, see \cite[Theorem 3.19]{boehm18_hopf_algeb_their_gener_categ}.

A monoidal category is \defTerm{left closed} if every \(\otimes\)-representable
functor \(\? \otimes X\) has a right adjoint, denoted \(\liHom{X}{\?}\).
 In this case, we call \(\liHom{X}{Z}\) the \defTerm{left internal Hom}
  from \(X\) to \(Z\).
Since
the adjunction
\((\eta^{X}, \? \otimes X \dashv \liHom{X}{\?}, \varepsilon^{X})\), gives a
biyection between \(\Hom{Y \otimes X}{Z}\) and \(\Hom{Y}{\liHom{ X}{Z}}\).
Borrowing the terminology from computer science, we define the
\defTerm{currying} of any \(\morf{f}{Y \otimes X}{Z}\) as
\(\curry{f} = \liHom{X}{f}\eta^{X}_{Y} \), and the \defTerm{uncurrying} of
\(\morf{g}{Y}{\liHom{X}{Z}}\) as
\(\uncurry{g} = \varepsilon^{X}_{Z} (g \otimes X)\).
It is worth noticing that by the theorem on adjunctions with a
  parameter --- see \cite[Theorem 3 in IV.7.]{MacLaneCategories} ---
  \(\liHom{\?_{1}}{\?_{2}}\) is functor from \(\mc{C}^{op} \otimes \mc{C}\) to
  \(\mc{C}\), when \(\mc{C}\) is left monoidal closed.

When the functor \(\liHom{X}{\?}\) is \(\otimes\)-representable,
i.e. \(\liHom{X}{\?} = \? \otimes \ldual{X}\) we call \(\ldual{X}\) the \defTerm{left dual} of \(X\).
A monoidal category is \defTerm{left rigid} (also called compact or autonomous)
if every object has a left dual.

Since a bimonad \(T\) lifts the monoidal product,
for every \(T\)-action \(M=(X,r_{M})\) the functor
\(\fc{\? \otimes M}{\mc{C}^{T}}{\mc{C}^{T}}\) is a lift
of \(\fc{\? \otimes X}{\mc{C}}{\mc{C}}\).
The bimonad \(T\) is a \defTerm{left Hopf monad} if whenever \(\? \otimes X\) has a right adjoint,
the adjoint lifts to an adjoint of \(\? \otimes M\).
Consequently, if \(\mc{C}\) is left monoidal closed (resp. rigid) and
\(T\) is a left Hopf monad, then \(\mc{C}^{T}\) is also left closed (resp. rigid)
and the forgetful functor preserves internal Homs.
Left Hopf monads are characterized by the invertibility of the left fusion operator,
see \cite[Theorem 3.6]{BrLaVi-HMMC} or
\cite[Theorem 3.27]{boehm18_hopf_algeb_their_gener_categ}.

Another key feature of monoidal categories is that,
in addition to the usual duality between
a category and its opposite, there is a left-right duality obtained by
exchanging the order in the monoidal product.
Implying that there are right variants for each of the above defined concepts.
When the laterality is omitted it is tacitly understood that both variants holds.
For instance, a rigid category is both left and right rigid.

A gadget for comparing left and right \(\otimes\)-representable functors
is called a braiding. More precisely,
a \defTerm{lax half-braiding} on a monoidal category \(\mc{C}\) is a pair \((A,\sigma)\) where $A$ is
  an object and \(\nt{\sigma}{A \otimes \?}{\? \otimes A}\) is a
  \(\otimes\)-\defTerm{multiplicative} natural transformation.
  Where \(\otimes\)-multiplicative means that \(\sigma_{\mb{1}} = \iid{A} \) and
  \[\sigma_{X \otimes Y} = (\id{X} \otimes \sigma_Y )(\sigma_X \otimes \id{Y}), \]
  for arbitrary \(X\) and \(Y\) in \(\mc{C}\).
  The adjective lax is dropped when \(\sigma\) is an isomorphism.
  A \defTerm{braiding} is a \(\otimes\)-multiplicative choice of a half braiding for every
  object in \(\mc{C}\).

  A \defTerm{lax central bimonoid} is simultaneously a lax central comonoid and monoid such that the multiplication and unit are comonoid homomorphisms.
  Explicitly, for a lax central bimonoid \(A = (A, \sigma, \delta, \mu, \varepsilon, \eta)\):
  \((A, \sigma)\) is a lax half-braiding,
  \(( A, \delta, \varepsilon )\) is a comonoid,
  \((A, \mu, \eta)\) is a monoid,
  and they satisfy \((\mu \otimes \mu)(A \otimes \sigma_{A} \otimes A)(\delta \otimes \delta) = \delta \mu\),
  \(\varepsilon \otimes \varepsilon = \varepsilon \mu\), \(\eta \otimes \eta = \delta \eta\) and \( \varepsilon \eta = \id{\mb{1}}\).
  Every lax central bimonoid defines a bimonad given by the functor \(A\otimes \?\),
  see \cite[Lemma 5.6]{BrLaVi-HMMC}.

  A lax central bimonoid \(H = (H, \sigma, \delta, \mu, \varepsilon, \eta)\) is
  called \defTerm{lax Hopf monoid} if the fusion morphism
  \((H \otimes \mu)(\delta \otimes H)\) is an isomorphism. By the the
  relation between convolution and fusion morfism,
  a lax Hopf monoid has an \defTerm{antipode}
  \(\morf{s}{H}{H}\) such that
  \[\mu (s \otimes H) \delta = \eta \varepsilon = \mu (H \otimes s) \delta.\]
  The adjective lax is dropped when both \(\sigma\) and \(s\) are isomorphisms.

  Since this paper is about augmented monads is very significant for us
  that in the Hopf case every augmented monad is representable.

\begin{thm}\label{thm:augmented-Hopf-monads}
  {\cite[Corollary 5.18]{BrLaVi-HMMC}}
  There is an equivalence between the categories of
  augmented left Hopf monads and lax central Hopf monoids,
  which restricts to an equivalence between augmented Hopf monads and
  central Hopf monoids.
\end{thm}

\subsection{Galois connections induced by functorial relations} \label{sec:galois-connections}

This is slight generalization of the Galois connection induced by a relation
relationship on sets. It is completely elementary and should be
folklore, but the author didn't find any appropriate reference and decided to
include brief but self contained exposition of this.

\begin{defi}\label{def:functorial-rel}
  We call \defTerm{functorial relation} between the categories \(\mc{X}\) and
  \(\mc{Y}\) to a functor \(\fc{R}{\mc{X} \times \mc{Y}}{\mc{2}}\), where
  \(\mc{2} = \{ \mathrm{true} \to \mathrm{false} \}\) is thought as the truth
  values of the relation.
  In other words, for objects \(X\) and \(Y\) in \(\mc{X}\) and \(\mc{Y}\), respectively,
  we say that \(X\) \defTerm{is related to} \(Y\) when \(R(X,Y) = \mathrm{true}\).

  Notice that asserting that \(R\) is a functor simply means that for any pair
  of morphisms, \(\morf{f}{X'}{X}\) and \(\morf{g}{Y'}{}Y\), if \(X\) is related
  to \(Y\) then \(X'\) is related \(Y'\).
\end{defi}

\begin{defi}\label{def:Universal-solution}
  The \(R\)-\defTerm{representant} for an object \(X\) in \(\mc{X}\) is an object, denoted \(R^{\ast} X\), such that for every \(Y\) in \(\mc{Y}\), \(X\) is related to \(Y\) if and only if there is an unique morphism \(Y\to R^{*}X\).
  The \(R\)-representant for a \(Y\) in \(\mc{Y}\), denoted \(R_{\ast}Y\) is analogously defined.
\end{defi}

The next definition generalizes the case of a Galois connection between
partially ordered sets, to adjunctions with a pre-order as its core.
Remember that a pre-order is a category with at most one morphism among any pair
of objects.
\begin{defi}\label{def:Galois-connection}
  A \defTerm{Galois connection}, between the categories \(\mc{X}\) and \(\mc{Y}\), is a contravariant adjunction \(\fc{G}{\mc{Y}}{\mc{X}^{op}} \dashv \fc{F}{\mc{X}^{op}}{\mc{Y}}\) such that
  \(\Hom[\mc{Y}]{Y}{FX} \approx \Hom[\mc{X}]{X}{GY}\) has at most one element, for arbitrary \(X\) and \(Y\).
\end{defi}

\begin{rem}\label{rem:Galois-connection}
  Let \(\fc{G}{\mc{Y}}{\mc{X}^{op}} \dashv \fc{F}{\mc{X}^{op}}{\mc{Y}}\) be a Galois connection.
  \begin{enumerate}[i.]
    \item For any \(X\) in \(\mc{X}\), there is a unique morphism
          \(X \to GF X\). We say that \(X\) is \defTerm{Galois closed} if this is
          an isomorphism. The Galois closed objects in \(\mc{X}\) form a
          pre-order.
    \item Since evaluating \(G\) at \(\morf{}{Y}{FGY}\) gives a morphism
          \(\morf{}{GFGY}{GY}\), for any \(Y\) in \(\mc{Y}\), \(GY\) is Galois
          closed in \(\mc{X}\). Analogously, for any \(X\) in \(\mc{X}\), \(FX\)
          is closed in \(\mc{Y}\).
    \item In conclusion, \(F\) and \(G\) induce an order-reversing equivalence
          between Galois closed objects. This is known as the induced
          \defTerm{Galois correspondance}.
  \end{enumerate}
\end{rem}

Now we are able to define the \defTerm{Galois connection induced by} \(R\).
\begin{prop}\label{prop:Galois-connection}
  Let \(R\) be a functorial relationship between \(X\) and \(Y\).
  \begin{enumerate}[i.]
    \item If every object in \(\mc{X}\) has a representant, then mapping
          \(X \mapsto R^{\ast}X\) defines a functor
          \(\fc{R^{\ast}}{\mc{X}^{op}}{\mc{Y}}\).

          Analogously, when every object in \(\mc{Y}\) has a representant there is \(\fc{R_{\ast}}{\mc{Y}}{\mc{X}^{op}}\).

    \item If both \(R^{\ast}\) and \(R_{\ast}\) are defined, then they form a Galois connection.
  \end{enumerate}
\end{prop}
\begin{proof}
  Given a morphism \(\morf{f}{X'}{X}\) in \(\mc{X}\), since \(R\) is functorial and \(X\) is related to \(R^{\ast}X\), then \(X'\) is related to  \(R^{\ast}X\); hence there is an unique
  \(R^{\ast}X \to R^{\ast} X' \).

  The adjunction plainly recalls that the existence of maps \(Y \to R^{\ast}X\) and  \(X \to R_{\ast}Y\) are both equivalent to \(X\) is related to \(Y\).
\end{proof}

\section{Invariants and stabilizers for augmented monads} \label{sec:Inv-Stab}

This section contains the main result of this paper. It states
that in an augmented monad it is possible to define of invariants
and stabilizers, in a way that naturally establish a Galois correspondence
when they exist.

\begin{defi}\label{def:Augment}
  Let \(T = (T, \mu, \eta)\) be a monad over \(\mc{C}\).
  An \defTerm{augmentation} of \(T\) is a monad homomorphism
  \(\nt{e}{T}{\Id{\mc{C}}}\). This is a natural transformation which satisfies
  \(e \cV \eta = \id{\Id{\mc{C}}}\) and \(e \cV \mu = e \cH e\).
  A monad endowed with an augmentation is called augmented monad.
\end{defi}

The intuition for an augmentation, is that the morphism \(\morf{e_{X}}{TX}{X}\) is the trivial action for any \(X\) in \(\mc{C}\); the action which forgets \(T\) and leaves \(X\) fixed.
During this whole section we will work in an augmented monad \(T = (T, \mu, \eta, e)\) over
a category \(\mc{C}\) and denote by \((\eta, F \dashv U, \varepsilon)\) the
Eilenberg-Moore decomposition of \(T\); where \(\fc{U}{\mc{C}^{T}}{\mc{C}}\) is
the forgetful functor from the category of \(T\)-actions, and its left adjoint
\(F\) sends every object in \(C\) to the free action on it. See definition
\ref{def:Eilenberg-Moore} for the details.

\subsection{The fix relation} \label{sec:trivial-actions}

Let \(\mC{Monad}_{\mc{C}}\) be the category of monads over \(\mc{C}\).
For brevity, we denote by \(S_{h}\) a monad homomorphism \(\nt{h}{S}{T}\)
considered as object of \(\mC{Monad}_{\mc{C}} / T\) and by \(V_{\alpha}\) a
natural transformation \(\nt{\alpha}{V}{U}\) from a functor
\(\fc{V}{\mc{C}^T}{\mc{C}}\) considered as object of
\(\mC{Fun} (\mc{C}^T , \mc{C}) / U\).

Since the augmentation provides us with a notion of trivial action, we can
compare any action with the trivial one over the same object. Moreover, we can
restrict this comparison to a part of the monad or to a part of the object. The
next definition does both restrictions simultaneously and then performs the
comparison.

\begin{defi} \label{def:fix}
  We say that \(S_{h}\) \defTerm{fixes} \(V_{\alpha}\)  when
  \(  (\id{U} \varepsilon) \cV (h \cH \alpha) = (e \id{U}) \cV (h \cH \alpha).\)
\end{defi}

Let's expand the above definition for a given \(T\)-action
\(M = (X,\morf{r_{M}}{TX}{X})\):

Since \(X = U M\) and the action morphism \(r_{M}\) equals
\(U \varepsilon_M\), then
the left hand side is obtained
by composing \(r_{M}\) with
\((h \cH \alpha)_{M} = (T\alpha_{M}) h_{VM} = h_{X} (S\alpha_{M})\).

On the other hand, \((e \id{U})_{M} = e_{X}\) is the trivial action;
therefore, when \(S_h\) fixes \(V_\alpha\) the restriction of
the trivial action equals the restricted action, this is
\[r_{M} h_{X} (S\alpha_{M}) = e_{X} h_{X} (S\alpha_{M}).\]

\begin{rem}\label{rem:Notation-monad-hom}
  By Corollary \ref{cor:Monad-hom-funtors}, every monad homomorphism
  \(\nt{h}{S}{T}\) corresponds to a functor \(\fc{H}{\mc{C}^{T}}{\mc{C}^{S}}\)
  such that \(U^{S} H = U\), where \(\fc{U^{S}}{\mc{C}^{S}}{\mc{C}}\) is the forgetful
  functor from the category of \(S\)-actions.
  In particular, the augmentation \(\nt{e}{T}{\Id{\mc{C}}}\)
  corresponds a functor \(\fc{E}{\mc{C}}{\mc{C}^T}\) which is
  a section of the forgetful functor, i.e. \(U E = \Id{\mc{C}}\).

  Recalling the definitions, any \(X\) in \(\mc{C}\) gets mapped to
  \(EX = (X, \morf{e_{X}}{TX}{X})\), and any \(M = (X, \morf{r_{M}}{TX}{X})\) in \(\mc{C}^{T}\)
  gets mapped to \(HM = (X, \morf{r_{M}h_{X}}{SX}{X})\).
\end{rem}

The next lemma allow us to reinterpret the definition of acting trivially in terms of the functors defined in the above remark.
\begin{lem}\label{lem:Fix-monadic-functors}
  The homomorphism \(\nt{h}{S}{T}\) fixes \(\nt{\alpha}{V}{U}\) if and only if there is a natural transformation
  \(\nt{\hat{\alpha}}{HEV}{H}\) such that \(U^{S} \hat{\alpha} = \alpha\).
\end{lem}
\begin{proof}
  First notice that \(H\) and \(H E V\) are a lifts of \(U\) and \(V\), respectively,
  along the monads \(\Id{\mc{C}^{T}}\) and \(S\).
 \[
\begin{tikzcd}
  {\mc{C}^T} && {\mc{C}^S} \\
  \\
  {\mc{C}^T} && \mc{C}
  \arrow[""{name=0, anchor=center, inner sep=0}, "U"', curve={height=8pt}, from=3-1, to=3-3]
  \arrow["{U^{S}}", from=1-3, to=3-3]
  \arrow[""{name=1, anchor=center, inner sep=0}, "V", shift left=2, curve={height=-6pt}, from=3-1, to=3-3]
  \arrow[""{name=2, anchor=center, inner sep=0}, "H"', curve={height=6pt}, from=1-1, to=1-3]
  \arrow["{\Id{\mc{C}^{T}}}",tail reversed, from=1-1, to=3-1]
  \arrow[""{name=3, anchor=center, inner sep=0}, "HEV", shift left=2, curve={height=-6pt}, from=1-1, to=1-3]
  \arrow["\alpha", shorten <=2pt, shorten >=2pt, Rightarrow, from=1, to=0]
  \arrow["\bar{\alpha}", shorten <=2pt, shorten >=2pt, Rightarrow, from=3, to=2]
\end{tikzcd}
\]

The lifting data for \(H\) over \(U\) is \((\id{U} \varepsilon) \cV (h \id{U}) \)
and for \(HEV\) over \(V\) is \((e \cV h) \id{V}\).

As an application of Proposition \ref{prop:lifts_nts}, there is a lift \(\bar{\alpha}\) if and only if
\(\alpha \cV((e \cV h) \id{V}) = (e \id{U}) \cV (h \cH \alpha) \)
equals
\((\id{U} \varepsilon) \cV (h \id{U}) \cV (\id{S} \alpha ) = (\id{U} \varepsilon)\cV (h \cH \alpha)\);
which is the definition of \(S_h\) fixes \(V_\alpha\).
\end{proof}

Now we will check that the fix relationship is preserved
whenever we restrict further.
In other words that fixing defines a functorial relation
between \(\mC{Monad}_{\mc{C}} / T\) and \(\mC{Fun} (\mc{C}^T , \mc{C}) / U\).

Afterwards, we will apply Definition \ref{def:Universal-solution} and Proposition \ref{prop:Galois-connection},
to this relation, in order to define a Galois connection.

\begin{prop}\label{prop:fix-functorial}
  Let \(\nt{g}{S'}{S}\) be any monad homomorphism and
  \(\nt{\beta}{V'}{V}\) be any natural transformation.
  If \(\nt{h}{S}{T}\) fixes \(\nt{\alpha}{V}{U}\), then \(\nt{h \cV g}{S'}{T}\) fixes \(\nt{\alpha \cV \beta}{V'}{U}\).
\end{prop}
\begin{proof}
  This is obvious from the facts that \((h \cV g) \cH (\alpha \cV \beta) = (h \cH \alpha)\cV(g \cH \beta)\) and the definition of fixing
  \((\id{U} \varepsilon) \cV (h \cH \alpha) = (e\id{U}) \cV (h \cH \alpha)\).
\end{proof}

\begin{defi}\label{def:Inv-Stab}
  Let \(\nt{h}{S}{T}\) be a monad homomorphism and \(\nt{\alpha}{V}{U}\) a natural transformation,
  where \(\fc{U}{\mc{C}^{T}}{\mc{C}}\) is the forgetful functor.
  \begin{enumerate}[(a).]
    \item We say that \(h\) is \defTerm{the stabilizer of} \(\alpha\) (or \(S\)
          is the stabilizer of \(V\)) if: it fixes \(\alpha\) and for every
          monad homomorphism \(\nt{h'}{S'}{T}\) that also fixes \(\alpha\) there
          is an unique \(\nt{g}{S'}{S}\) such that \(h'= h \cV g\).
    \item We say that \(\alpha\) is the \(h\)-\defTerm{invariant inclusion} (or
          \(V\) is the functor of \(S\)-invariants) if: it is fixed by \(h\) and
          for every functor \(\fc{\alpha'}{V'}{U}\) also fixed by \(h\) there is
          an unique \(\fc{\beta}{V'}{V}\) such that
          \(\alpha' = \alpha \cV \beta\).
  \end{enumerate}
\end{defi}

The next corollary is an straightforward application of Proposition \ref{prop:Galois-connection}, to the fixing relationship.
\begin{cor}\label{cor:Galois-corr}
  Let \(\mc{M}\) and \(\mc{N}\) be subcategories of \(\mC{Monad} (\mc{C}) / T\)
  and \(\mC{Fun} (\mc{C}^T , \mc{C}) / U\), respectively.
  \begin{enumerate}[i.]
    \item If every \(h\) in \(\mc{M}\) has an \(h\)-invariant inclusion in
          \(\mc{N}\), then the mapping \(h \mapsto \mFc{Inv} h\) defines a
          functor \(\fc{\mFc{Inv}}{\mc{M}^{op}}{\mc{N}}\). Analogously, If every
          \(\alpha\) in \(\mc{N}\) has an stabilizer in \(\mc{M}\), then the
          mapping \(\alpha \mapsto \mFc{Stab} \alpha\) defines a functor
          \(\fc{\mFc{Stab}}{\mc{N}}{\mc{M}^{op}}\).

  \item If both $\mFc{Stab}$ and $\mFc{Inv}$ are defined, then they form a Galois connection.
  \end{enumerate}
\end{cor}

\section{Computing invariants via right adjoints} \label{sec:Invariants-as-adjoints}

Now that we have an abstract definition of invariants and stabilizers, we face
the concrete --- or slightly less abstract --- questions of determining when
they exist and how can we compute it. This segment offers an answer for the case
of invariants, by linking them to right adjoints.
Remember that \((T, \mu, \eta, e)\) is an augmented monad over \(\mc{C}\).

First we observe that when the domain monad has a right adjoint we compute
invariants as an equalizer.
\begin{lem} \label{lem:inv-monad-adjoint}
  Let $\nt{h}{S}{T}$ be a monad homomorphism and assume that the monad
  \(\fc{S}{\mc{C}}{\mc{C}}\) has a right adjoint,
  \((\theta, S \dashv R, \upsilon)\).

  For any \(\nt{\alpha}{V}{U}\), \(S_h\) fixes \(V_\alpha\) if and only if for
  each \(M=(X,r_{M})\) in \(\mc{C}^{T}\),
  \((R(eh)_{X}) (\theta_{X}) \alpha_{M} = (R(r_{M}h_{X}))(\theta_{X}) \alpha_{M} \).

  In consequence, when the pair \((R(eh)_{X}) (\theta_{X})\) and
  \((R(r_{M}h_{X}))(\theta_{X})\) has an equalizer for every \(M\), this
  equalizer defines the functor of \(S_h\)-invariants.
\end{lem}
\begin{proof}
  Just recall that \(S_h\) fixes \(V_\alpha\) when
  \(r_{M} h_{X} (S\alpha_{M}) = e_{X} h_{X} (S\alpha_{M})\) for every \(M\), and
  bend \(S\) to \(R\) in both sides of the equation.
In other words,
\[(R ((U \varepsilon) \cV (h U))) \cV (\theta U) \cV \alpha = (R(e \cV h)U) \cV (\theta U) \cV \alpha .\]
Since equalizers are funtorial, the last sentence is evident.
\end{proof}

The next example illustrate the above lemma in the case of a \(\otimes\)-representable monad.
\begin{ex}
  Let \(T = A \otimes \?\) for an augmented monoid
  \[A = (A, \morf{m}{A \otimes A}{A}, \morf{u}{\mb{1}}{A}, \morf{\epsilon}{A}{\mb{1}})\]
  in a monoidal category \((\mc{C}, \otimes, \mb{1})\).
  Assume that \(A\) is right closed in \(\mc{C}\)
  and that \(\mc{C}\) has equalizers.
  The functor of \(A\)-invariants \(\mFc{Inv} A := \mFc{Inv} T\)
  sends each action \(M=(X, \morf{r_{M}}{A \otimes X}{X})\)
  to the equalizer of
\(
\begin{tikzcd}[cramped]
  {X} & {\riHom{A}{X}.}
  \arrow["{\curry{r_{M}}}", shift left=1.5, from=1-1, to=1-2]
  \arrow["{\curry{\epsilon \otimes X}}"', shift right=1, from=1-1, to=1-2]
\end{tikzcd}
\)

Moreover, given a monoid homomorphism \(\morf{h}{B}{A}\),
with \(B\) also right closed in \(C\).

  The functor of \(B\)-invariants \(\mFc{Inv} B := \mFc{Inv} (B\otimes \?)_{h\otimes \?}\)
  sends each action to the equalizer of
\(
\begin{tikzcd}[cramped]
  {X} & {\riHom{A}{X}} & {\riHom{B}{X}}.
  \arrow["{\curry{r_{M}}}", shift left=1.5, from=1-1, to=1-2]
  \arrow["{\curry{\epsilon \otimes X}}"', shift right=1, from=1-1, to=1-2]
  \arrow["{\riHom{h}{X}}", from=1-2, to=1-3]
\end{tikzcd}
  \)
\end{ex}

\subsection{Invariants are right adjoints} \label{sec:invariants-are-right}

Right adjoints not only allow us to compute invariants.
In fact, the \(T\)-invariants,
i.e the invariants for \(\nt{\iid{T}}{T}{T}\),
are right adjoints of the trivial action functor \(E\).
Where \(\fc{E}{\mc{C}}{\mc{C}}^{T}\) is the functor corresponding to the
augmentation \(\nt{e}{T}{\Id{\mc{C}}}\), remember that it endows
every object in \(\mc{C}\) with the trivial action given by the augmentation.

\begin{lem}\label{lem:Inv-adjoint-E}
  If \(E\) has a right adjoint functor \(\fc{\Gamma}{\mc{C}^{T}}{\mc{C}}\),
  then \(\Gamma\) is the functor of \(T\)-invariants.
\end{lem}
\begin{proof}
Take the adjunction \((\zeta, E \dashv \Gamma, \bar{\gamma})\) and define \(\nt{\gamma}{\Gamma}{U}\) as \(\gamma = U \bar{\gamma}\).
  Hence by Lemma \ref{lem:Fix-monadic-functors}, \(\id{T}\) fixes \(\Gamma_\gamma\).
  Moreover by the same lemma, if \(\id{T}\) fixes \(\nt{\alpha}{V}{U}\)
  there is \(\nt{\bar{\alpha}}{EV}{\Id{\mc{C}^{T}}}\) such that \(\alpha = U\bar{\alpha}\).
  We need a \(\nt{\beta}{V}{\Gamma}\) such that \(\alpha = \gamma \cV \beta = (U \bar{\gamma}) \cV (UE \beta)\) since \(U\) is faithful  \(\bar{\alpha} =  \bar{\gamma} \cV (E \beta)\),
  after bending the \(E\)
  we obtain that \((\Gamma \bar{\alpha}) \cV (\zeta V) = \beta\).

  Since \(E = (\bar{\gamma} E)\cV (E \zeta)\) then
  \begin{align*}
    \bar{\alpha} &= \bar{\alpha}\cV (\bar{\gamma} EV)\cV (E \zeta V)\\
                 &= \bar{\gamma}\cV (E \Gamma \bar{\alpha} )\cV (E \zeta V).
  \end{align*}
  Therefore \(\alpha = U \bar{\alpha} = \gamma \cV (\Gamma \bar{\alpha}) \cV (\zeta V)\).
\end{proof}

Right adjoints also allow to change the category where one is calculating.
Let $\nt{h}{S}{T}$ be a monad homomorphism, with correspondent functor
$\fc{H}{\mc{C}^{T}}{\mc{C}^S}$ as in Remark \ref{rem:Notation-monad-hom}; in particular \(U = U^{S} H\).

Notice that \((S, \mu', \eta', eh)\) is also an augmented monad and
instead of computing $h$-invariants in $\mc{C}^T$, we can try to compute
$\id{S}$-invariants in $\mc{C}^S$ and pull them back.
The next two lemmas show that this strategy is successful when \(H\) has a right adjoint.

\begin{lem}\label{lem:Fix-adjoint}
  Let \((\zeta, H \dashv K, \omega)\) be an adjunction
  with \(\fc{H}{\mc{C}^{T}}{\mc{C}^{S}}\),
  coming from a monad homomorphism \(\nt{h}{S}{T}\).
  For any \(\nt{\alpha}{V}{U}\), \(S_h\) fixes \(V_\alpha\) is equivalent to
  \(\id{S}\) fixes \(\nt{\hat{\alpha}}{VK}{U^{S}}\),
  where \(\hat{\alpha} = (U^{S} \omega) \cV (\alpha K)\).
\end{lem}
\begin{proof}
  We will use Lemma \ref{lem:Fix-monadic-functors}.
  If \(S_h\) fixes \(V_\alpha\), there is \(\nt{\bar{\alpha}}{HEV}{H}\) such that \(\alpha = U^{S}\bar{\alpha}\).
  Therefore
  \[\hat{\alpha} = (U^{S} \omega) \cV (\alpha K)  = U^{S}(\omega \cV (\bar{\alpha} K)).\]
  Hence, by the same lemma \(\id{S}\) fixes \((VK)_{\hat{\alpha}}\).

  For the reciprocal, notice that \(\alpha = (\hat{\alpha} H)\cV (V \zeta)\) by bending \(K\) to \(H\).
  Now the proof is completely analogous:

  If \(\id{S}\) fixes \(\hat{\alpha}\), there is \(\nt{\tilde{\alpha}}{HEVK}{\Id{\mc{C}^{S}}}\)
  such that \(\hat{\alpha} = U^{S} \tilde{\alpha}\).
  Therefore
  \[\alpha = (\hat{\alpha} H) \cV (V \zeta)  = U ^{S}((\tilde{\alpha} H) \cV (HEV \zeta)),\]
  since \(U^{S} H E = UE = \Id{\mc{C}}\).
  The same lemma tells us that \(S_h\) fixes \(V_\alpha\).
\end{proof}

\begin{cor}\label{cor:Inv-adjoint-change}
  Let \(\nt{h}{S}{T}\) be a monad homomorphism and suppose that
  \(\fc{H}{\mc{C}^{T}}{\mc{C}^{S}}\) has a right adjoint. If
  \(\nt{\gamma^{S}}{\Gamma^{S}}{U^{S}}\) is the  $\id{S}$-invariant inclusion, then
  \(\nt{\gamma^{S} \id{H}}{\Gamma^{S} H}{U}\) is the $h$-invariant inclusion.
\end{cor}

\begin{proof}
  With the notations from the previous lemma,
  \(S_h\) fixes \(V_\alpha\) if and only
  if \(\id{S}\) fixes \(\hat{\alpha}\).
  Since \(\gamma^{S}\) is the \(S\)-invariant inclusion, there is an unique
  \(\nt{\hat{\beta}}{\Gamma^{S}}{VK}\) such that
  \(\hat{\alpha} = \gamma^{S} \cV \hat{\beta}\).
  Remembering that \(\alpha = (\hat{\alpha} H)\cV (V \omega)\) we obtain
  \(\alpha = ((\gamma^{S} \cV \hat{\beta}) H)\cV (V \omega)= (\gamma^{S}H)\cV \beta \)
  where \(\beta = (\hat{\beta} H)\cV (V \omega)\).
\end{proof}

We end this segment with a summary of the computation of invariants
by means of right adjoints.

\begin{cor}
  Let \((T, \mu, \eta, e)\) be an augmented monad,
  and \(\nt{h}{S}{T}\) a monad homomorphism.
  Take \(\fc{E}{\mc{C}}{\mc{C}}^{T}\) and \(\fc{H}{\mc{C}^{S}}{\mc{C}^{T}}\) the functors corresponding to
  \(e\) and \(h\), respectively;
  and assume that they are part of adjunctions
  \((\zeta, E \dashv \Gamma, \bar{\gamma})\) and \((\zeta, H \dashv K, \omega)\).

  Then \(\mFc{Inv} T = \Gamma_{\gamma} \) with \( \gamma = U \bar{\gamma}\),
  \(\mFc{Inv} S = \Gamma K_{\hat{\gamma}}\) with
  \(\hat{\gamma} = (U^{S} \omega) \cV (U \bar{\gamma} K),\)
  and \(\mFc{Inv} S_h =  \Gamma K H _{\hat{\gamma}H}\)
  with \({\hat{\gamma}H} = (U^{S} \omega H) \cV (U \bar{\gamma} KH)\).

\end{cor}
\begin{proof}
  This is Lemma \ref{lem:Inv-adjoint-E} applied twice to \(E \dashv \Gamma\) and
  to \(H E \dashv  \Gamma K\), and Corollary \ref{cor:Inv-adjoint-change}.
  Notice that when \(H\) has also a left adjoint \(L\) then
  \( L H E \dashv \Gamma K H \).
\end{proof}

\section{Computing stabilizers in monoidal closed categories} \label{sec:Stab-mono}

Currently we don't know of any procedure for explicitly
describing the stabilizers for augmented monads in
categories without additional structure.
We imagine a description in terms of Kan extensions,
but we have not succeeded in formalizing it.
However, in the case of a \(\otimes\)-representable monad over a
monoidal closed category,
we can use \defTerm{ends} as a tool for the computation of stabilizers.
This is a procedure inpired by Tannakian reconstruction.
encompass augmented Hopf monads in monoidal closed categories.

\subsection{The end enrichement of functor categories} \label{sec:enrichement}

For this section, we assume familiarity with the basic properties of ends, as presented for instance in \cite[Sections IX. 5-8]{MacLaneCategories}.
For fixing notation, we briefly recall the definition.

\begin{defi}
  A \defTerm{wedge} for a functor \(\fc{S}{\mc{D}^{op} \times \mc{D}}{\mc{C}}\) is a pair \((W, \xi_{X})_{X\in \mc{D}}\)
  consisting of an object \(W\) in \(\mc{C}\) and
  a family of morphisms \(\morf{\xi_{X}}{W}{S(X,X)}\) indexed by the objects in \(\mc{C}\)
  such that for any morphism \(\morf{f}{X}{Y}\) in \(\mc{D}\) the diagram
  \[\begin{tikzcd}
  W & {S(X,X)} \\
  {S(Y,Y)} & {S(X,Y)}
  \arrow["{\xi_X}", from=1-1, to=1-2]
  \arrow["{S(\id{X},f)}", from=1-2, to=2-2]
  \arrow["{\xi_Y}"', from=1-1, to=2-1]
  \arrow["{S(f,\id{Y})}"', from=2-1, to=2-2]
\end{tikzcd}
\]
  commutes.

  The \defTerm{end} of \(S\) denoted
  \[\int_{\mc{D}} S =\int_{X \in \mc{D} } S(X,X) = (L, \lambda_{X})_{X \in \mc{D}} \]
  is the universal wedge. This means that for any other wedge, \((W,\xi_{X})_{X \in \mc{D}}\) there is an unique morphism, denoted \(\morf{\int_{X} \xi_{X}}{W}{L}\), such that \(\lambda_{Y}\int_{X} \xi_{X} = \xi_{Y}\) for any \(Y\) in \(\mc{C}\).
\end{defi}

The next definition allow us to view the natural transformations
as an object in the co-domain category.

\begin{defi} \label{def:Internal-Nat}
  Let \(V\) and \(W\) be two functors from an arbitrary category \(\mc{D}\) to
  a monoidal category \((\mc{C}, \otimes, \mb{1})\).
  Assume that \(\mc{C}\) is left closed over \(V\),
  i.e. for every \(D\) in \(\mc{D}\) the functor \(\fc{\? \otimes VD}{\mc{C}}{\mc{C}}\) has a right adjoint \(\liHom{VD}{\?}\).

  The \defTerm{left internal Nat} from \(V\) to \(W\) is defined as the end of
  the functor \(\fc{\liHom{V\?_{1}}{W\?_{2}}}{\mc{\mc{D}^{op}} \times \mc{D}}{\mc{\mc{C}}}\), and denote it by
  \[ \liHom{V}{W} := \int_{\mc{D}} \liHom{V\?_{1}}{W\?_{2}} = \int_{D \in \mc{D}} \liHom{V D}{W D} = (\liHom{V}{W}, \morf{\lambda_{D}}{\liHom{V}{W}}{\liHom{VD}{WD}})_{D \in \mc{D}}.\]
\end{defi}

\begin{rem} \label{rem:Functor-enriched}
  If the internal Nat \(\liHom{V}{W}\) exists for any pair of functors,
  --- for instance when \(\mc{C}\) is complete, left closed and \(\mc{D}\) is small ---
  then \(\mC{Fun} (\mc{D}, \mc{C})\) is a left \(\mc{C}\)-enriched category.
  The next two propositions are consequences of this fact.
\end{rem}

\begin{prop}\label{prop:Nats-bifunctor}
  Let \(V\), \(V'\), \(W\) and \(W'\) be functors from \(\mc{D}\) to \(\mc{C}\)
  with \(\nt{\omega}{W}{W'}\) and \(\nt{\upsilon}{V'}{V}\) natural transformations.
  \begin{itemize}
    \item If both internal Nats \(\liHom{V}{W}\) and \(\liHom{V}{W'}\) exist, then there is a morphism
    \[\morf{\liHom{V}{\omega}}{\liHom{V}{W}}{\liHom{V}{W'}}.\]

    \item Analogously, when \(\liHom{V}{W}\) and \(\liHom{V'}{W}\) exist, there is
    \[\morf{\liHom{\upsilon}{W}}{\liHom{V}{W}}{\liHom{V'}{W}}.\]

    \item Whenever the four internal Nats in the following diagram exist, the arrows also exist and the diagram is commutative.
    \[\begin{tikzcd}
  {\liHom{V}{W}} & {\liHom{V}{W'}} \\
  {\liHom{V'}{W}} & {\liHom{V'}{W'}.}
  \arrow["{\liHom{\upsilon}{W}}"', from=1-1, to=2-1]
  \arrow["{\liHom{V'}{\omega}}", from=2-1, to=2-2]
  \arrow["{\liHom{V}{\omega}}", from=1-1, to=1-2]
  \arrow["{\liHom{\upsilon}{W'}}", from=1-2, to=2-2]
\end{tikzcd}
\]
  \end{itemize}
\end{prop}
\begin{proof}
  In the first case, for any \(D\) in \(\mc{D}\) the composition of \(\morf{\lambda_{D}}{\liHom{V}{W}}{\liHom{VD}{WD}}\)
  and \(\morf{\liHom{VD}{\omega_{D}}}{\liHom{VD}{WD}}{\liHom{VD}{W'D}}\) defines a wedge for the
  functor \(\liHom{V\?_{1}}{W' \?_{2}}\).
  Define \( \liHom{V}{\omega} := \int_{D \in \mc{D}} \liHom{VD}{\omega_{D}} \lambda_{D}\).
\end{proof}

\begin{prop}\label{prop:Ends-composition}
  Take \(\fc{W}{\mc{D}}{\mc{C}}\) such that the internal Nat \(\liHom{W}{W}\) exist.
  There is a canonical monoid structure on \(\liHom{W}{W}\) such that
  for any \(D\) in \(\mc{D}\) the morphism \(\morf{\lambda_{D}}{\liHom{W}{W}}{\liHom{WD}{WD}}\) is a monoid homomorphism.
\end{prop}
\begin{proof}
  Remember that for any  \(D\) in \(\mc{D}\), \(\liHom{WD}{WD}\) is a monoid
  denoted \(\mFc{End} \ldual{[WD]}\).
  The units \((\morf{\curry{\iid{WD}}}{\mb{1}}{WD})_{D \in \mc{D}}\) make a wedge,
  therefore take \(\eta = \int_{D \in \mc{D}}\curry{\iid{WD}}\) as unit for \(\liHom{W}{W}\).

  For the multiplication, compose
  \(\morf{\lambda_{D} \otimes \lambda_{D}}{\liHom{W}{W}\otimes \liHom{W}{W}}{\liHom{WD}{WD} \otimes \liHom{WD}{WD}}\)
  with \(\morf{\mu_{D}}{\mFc{End} \ldual{[WD]} \otimes\mFc{End} \ldual{[WD]} }{\mFc{End} \ldual{[WD]}}\),
  verify that this conforms a wedge\footnote{This uses that the compositions are associative and that \(\lambda_{D}\) is a wedge.}
  and define
  \( \mu = \int_{D \in \mc{D}} \mu_{D}(\lambda_{D} \otimes \lambda_{D})\).
  Finally, the monoid axioms are verified component-wise.
\end{proof}

Following the usual terminology we will denote \(\liHom{W}{W}\) by
\(\mFc{End} \ldual{[W]}\) and call it the \defTerm{left internal End} of the
functor \(W\). The next proposition formalizes the idea of restricting the
endomorphisms of \(W\) along a functor \(\fc{G}{\mc{D}'}{\mc{D}}\).

\begin{prop} \label{prop:Ends-restriction} Take \(\fc{W}{\mc{D}}{\mc{C}}\) and
  \(\fc{G}{\mc{D}'}{\mc{D}}\) such that \(\mFc{End} \ldual{[W]}\) and
  \(\mFc{End} \ldual{[WG]}\) both exist, then there is a canonical monoid
  homomorphism \(\morf{|G}{\mFc{End} \ldual{[W]}}{\mFc{End} \ldual{[WG ]}}\).
\end{prop}
\begin{proof}
The map \(|G\) is defined by \(|G:= \int_{X \in \mc{D}'} \lambda_{GX}\) when \(\mFc{End} \ldual{[W]}= (\mFc{End} \ldual{[W]}, \lambda_{D})_{D \in \mc{D}}\).
\end{proof}

\begin{lem} \label{lem:Ends-universal-action} Take \(\fc{W}{\mc{D}}{\mc{C}}\)
  such that \(\mFc{End} \ldual{[W]}\) exists. There is a unique functor
  \(\fc{\overline{W}}{\mc{D}}{\mc{C}}^{\mFc{End} \ldual{[W]}}\) which lifts
  \(W\) to the category of actions of \(\mFc{End} \ldual{[W]}\). This means that
  \(W = U \overline{W}\) where
  \(\fc{U}{\mc{C}^{\mFc{End} \ldual{[W]}}}{\mc{C}}\) is the forgetful functor.
\end{lem}
\begin{proof}
  For \(D\) in \(\mc{D}\) define
  \(\overline{W} D = (WD, \uncurry{\lambda_{D}})\) where
  \(\morf{\lambda_{D}}{\mFc{End} \ldual{[W]}}{\liHom{WD}{WD}}\) and
  \(\morf{\uncurry{\lambda_{D}}}{\mFc{End} \ldual{[W]} \otimes WD}{WD}\) is
  obtained by un-currying.
\end{proof}

\subsection{Tannakian reconstruction of stabilizers} \label{sec:reconstr}

The above construction of internal Nats gives us a path for reconstruction. Let

Let \((\mc{\bar{C}}, \otimes, \mb{1})\) be a monoidal category,
with all small limits and
left closed over a small subcategory \(\mc{C} \hookrightarrow \mc{\bar{C}}\),
i.e.
for every \(X\) in \(\mc{C}\) the functor \(\fc{\? \otimes X}{\mc{\bar{C}}}{\mc{\bar{C}}}\)
has a right adjoint \(\liHom{X}{\?}\).
In this case, we denote by \(\mFc{End} \ldual{[\mc{C}]}\) the internal End for the
inclusion functor \(\mc{C} \hookrightarrow \mc{\bar{C}}\)
which exists in \(\mc{\bar{C}}\).
Moreover, for a small category \(D\)
every pair of functors \(\fc{V,W}{\mc{D}}{\mc{C}}\),
after composing with the inclusion \(\mc{C} \hookrightarrow \mc{\bar{C}}\)
has an internal Nat in \(\mc{\bar{C}}\),
which abusing a little of notation we call \(\liHom{V}{W}\).

From a more general perspective,
the next proposition defines a unit for the
Tannakian reconstruction adjunction.

\begin{prop} \label{prop:reconstruction-unit}
  Let \((A, m, u)\) be a monoid in \(\mc{\bar{C}}\), \(\mc{C}^{A} \hookrightarrow \mc{\bar{C}}^{A}\)
  the subcategory of \(A\)-actions on \(\mc{C}\) and
  \(\fc{U}{\mc{C}^{A}}{\mc{C}}\) the forgetful functor.
  There is a monoid homomorphism \(\morf{\rho}{A}{\mFc{End} \ldual{[U]}}\)such that
  the functors \(\fc{\overline{U}}{\mc{C}^{A}}{\mc{\bar{C}}}^{\mFc{End} \ldual{[U]}}\), obtained by \ref{lem:Ends-universal-action},
  and \(\fc{K}{\mc{\bar{C}}^{\mFc{End} \ldual{[U]}}}{\mc{\bar{C}}^{A}}\), corresponding to the monad homomorphism \(\rho \otimes \?\),
  compose to the inclusion \(\mc{C}^{A} \hookrightarrow \mc{\bar{C}}^{A}\).
\end{prop}
\begin{proof}

  For every \(A\)-action
  \(M = (X, (\morf{r_{M}}{A\otimes X}{X}))\) curry the action.
  Notice that for any \(A\)-action
  morphism to obtain \(\morf{\curry{r_{M}}}{A}{\liHom{X}{X}}\).
  homomorphism \(\morf{f}{M}{N}\) the diagram
 \[
\begin{tikzcd}
    A & {\liHom{ UM } { UM }} \\
    {\liHom{ UN } { UN }} & {\liHom{ UN } { UM }}
    \arrow["{\curry{r_N}}"', from=1-1, to=2-1]
    \arrow["{\liHom{ f } { UN }}", from=2-1, to=2-2]
    \arrow["{\curry{r_M}}", from=1-1, to=1-2]
    \arrow["{\liHom{ UM } { f  }}", from=1-2, to=2-2]
\end{tikzcd}
\]
commutes because
\(\curry{fr_{M}} = \liHom{X}{f} \curry{r_{M}}\) and
\(\curry{r_{N}(Z \otimes f) } = \liHom{f}{Y} \curry{r_{N}} \).

Therefore
\((A, \curry{r_{M}})_{M \in \mc{C}^{A}}\) is a wedge, and we obtain a morphism
\(\morf{\rho}{A}{\mFc{End} \ldual{[U]}}\) such that \(\curry{r_{M}} = \lambda_{M} \rho\).
Since \(\overline{U} M = (X, \uncurry{\lambda_{M}})\) then \(K \overline{U} M = (X, \uncurry{\lambda_{M}} (\rho \otimes X))\),
and \(\uncurry{\lambda_{M}} (\rho \otimes X) = \uncurry{\lambda_{M} \rho} = r_{M}\).
\end{proof}

Now we turn to the question of internal Ends which fix a natural transformation.
\begin{prop} \label{prop:internal-augmentation}
  Let \((A, m, u, \epsilon)\) be an augmented monoid in \(\mc{\bar{C}}\).
  Then we have a diagram of monoid homomorphisms
\[
  \begin{tikzcd}
    {A} & {\mFc{End} \ldual{[U]}} \\
    {\mb{1}} & {\mFc{End} \ldual{[\mc{C}]}},
    \arrow["{\epsilon}"', from=1-1, to=2-1]
    \arrow["{\rho}", from=1-1, to=1-2]
    \arrow["{\eta}"', from=2-1, to=2-2]
    \arrow["{|E}"', shift right=2, from=1-2, to=2-2]
    \arrow["{|U}"', shift right=2, from=2-2, to=1-2]
  \end{tikzcd}
\]
where \((|E)(|U) = \iid{\mFc{End} \ldual{[\mc{C}]}}\) and
\((|U) \eta \epsilon = \int_{M \in \mc{C}^{A}} \curry{\epsilon \otimes UM}\).
\end{prop}

\begin{proof}
  Notice that the functor \(\fc{E}{\mc{C}}{\mc{C}^{T}}\),
  given by \(EX = (X,\morf{\epsilon\otimes X}{A\otimes X}{X})\) is a section for \(U\).
  The morphisms \(|E\) and \(|U\) are given by Proposition \ref{prop:Ends-restriction},
  and \((|U)(|E) = |(UE) = \iid{\mFc{End} \ldual{[\mc{C}]}} \).
  Lets check that \((|E) \rho = \eta \epsilon\):
  For any \(X\) in \(\mc{C}\), consider the universal wedge components \(\morf{\xi_{X}}{\mFc{End} \ldual{[C]}}{\liHom{X}{X}}\)
  and \(\morf{\lambda_{M}}{\mFc{End} \ldual{[U]}}{\liHom{UM}{UM}}\).
  On one hand \(\xi_{X}(|E) \rho = \lambda_{EX} \rho = \curry{r_{EX}} = \curry{\epsilon \otimes X}\) since \(\rho = \int_{M} \curry{r_{M}}\).
  On the other hand \(\xi_{X} \eta \epsilon = \curry{\iid{X}} \epsilon = \curry{\epsilon \otimes X}\) since \(\eta = \int_{X} \curry{\iid{X}}\).
\end{proof}
\begin{lem} \label{lem:Stab-as-equal}
  Take a monoid homomorphism \(\morf{h}{B}{A}\) in \(\mc{\bar{C}}\)
  and a natural transformation \(\nt{\alpha}{V}{U}\), then
  \((B\otimes \?)_{h \otimes \?}\) fixes \(V_{\alpha}\) if and only if
  \(h\) equalizes the diagram
\(
\begin{tikzcd}[cramped]
  {A} & {\mFc{End} \ldual{[U]}} & {\liHom{V}{U}}.
  \arrow["{\rho}", shift left=1.5, from=1-1, to=1-2]
  \arrow["{(|U) \eta \epsilon }"', shift right=1, from=1-1, to=1-2]
  \arrow["{\liHom{\alpha}{U}}", from=1-2, to=1-3]
\end{tikzcd}
  \)
\end{lem}
\begin{proof}
  For any \(M = (X, r_{M})\) in \(\mc{C}^{T}\), consider the universal wedge components
  \(\lambda_{M}\) and \(\xi_{M}\) displayed in the next diagram:
 \[\begin{tikzcd}
  B & A &{\mFc{End} \ldual{[U]}}& {\liHom{UM}{UM}} \\
  && {\liHom{V}{U}} & {\liHom{VM}{UM}}
  \arrow["{(|U) \eta \epsilon }"', shift right=1, from=1-2, to=1-3]
  \arrow["{ \liHom{\alpha}{U} }"', from=1-3, to=2-3]
  \arrow["{\xi_M}", from=2-3, to=2-4]
  \arrow["{\lambda_{M}}", from=1-3, to=1-4]
  \arrow["{ \liHom{\alpha_{M}}{UM} }", from=1-4, to=2-4]
  \arrow["\rho", shift left=1, from=1-2, to=1-3]
  \arrow["h", from=1-1, to=1-2]
  \end{tikzcd}\]
On one hand, since \(\rho = \int_{M} \curry{r_{M}}\)
we obtain
\begin{align*}
  \xi_{M} \liHom{\alpha}{U} \rho h
  &= \liHom{\alpha_{M}}{UM} \lambda_{M} \rho h
    = \liHom{\alpha_{M}}{UM} \curry{r_{M}} h  \\
  &= \curry{r_{M}(h \otimes \alpha_{M})}.
\end{align*}
On the other hand, since
\((|U) \eta \epsilon = \int_{M \in \mc{C}^{A}} \curry{\epsilon \otimes UM}\),
we obtain
\begin{align*}
  \xi_{M} \liHom{\alpha}{U} (|U) \eta \epsilon   h
  &= \liHom{\alpha_{M}}{UM} \lambda_{M} (|U) \eta \epsilon h
    =  \liHom{\alpha_{M}}{UM} \curry{\epsilon \otimes X} h \\
  &= \curry{(\epsilon \otimes UM)(h \otimes \alpha_{M})}.
\end{align*}
The result follows because uncurrying the last term in each equation gives the definition of the fix relation.
\end{proof}

\bibliographystyle{amsplain}
\bibliography{main}
\end{document}